\documentclass[%
12pt%
,a4paper%
]{article}

\usepackage{amssymb}
\usepackage{amsmath}
\usepackage[normalem]{ulem}
\usepackage{amsthm}
\usepackage[utf8]{inputenc}
\usepackage[english]{babel}
\usepackage{graphicx}

\usepackage[a4paper,margin=2.5cm]{geometry}
 
\usepackage{enumerate}
\usepackage{url}
\usepackage{hyperref}
\usepackage[font=
footnotesize%
]{caption}

\theoremstyle{plain}
\newtheorem{theorem}{Theorem}[section]
\newtheorem{proposition}[theorem]{Proposition}
\newtheorem{lemma}[theorem]{Lemma}
\newtheorem{corollary}[theorem]{Corollary}

\theoremstyle{definition}

\newtheorem{assumption}[theorem]{Assumption}
\newtheorem{remark}[theorem]{Remark}
\newtheorem{example}[theorem]{Example}

\numberwithin{equation}{section}

\long\def\MSC#1\EndMSC{\def\arg{#1}\ifx\arg\empty\relax\else
     {\narrower\noindent%
{2010 Mathematics Subject Classification}: #1\\} \fi}
\long\def\KEY#1\EndKEY{\def\arg{#1}\ifx\arg\empty\relax\else
	{\narrower\noindent%
Key words: #1\\}\fi}

\newcommand{\abs}[1]{\lvert{#1}\rvert}  
\newcommand{\norm}[1]{\lVert{#1}\rVert} 
\newcommand{\ip}[2]{\langle{#1},{#2}\rangle} 

\DeclareMathOperator{\re}{Re} 
\DeclareMathOperator{\ran}{Ran} 
\DeclareMathOperator{\tr}{tr}

\newcommand{\IC}{\mathbb{C}}   
\newcommand{\IR}{\mathbb{R}}   
\newcommand{\IN}{\mathbb{N}}   
\newcommand{\IZ}{\mathbb{Z}}   
\newcommand{\IT}{\mathbb{T}}   
\newcommand{\I}{\mathrm{i}}    
\newcommand{\e}{\mathrm{e}}    
\newcommand{\di}{\mathrm{d}}   
\newcommand{\dist}{\mathrm{dist}}     

\newcommand{\cH}{\mathcal{H}}
\newcommand{\cG}{\mathcal{G}}

\DeclareMathOperator{\diag}{diag}

\begin{document}
\title{A local directional growth estimate of the resolvent norm}
\author{
H.~D.~Cornean\footnote{Department of Mathematical Sciences, Aalborg University, Skjernvej 4A, 9220 Aalborg \O, Denmark (cornean@math.aau.dk, henrik@math.aau.dk, matarne@math.aau.dk, hanskk@math.aau.dk).} ,
H.~Garde\footnotemark[1] , 
A.~Jensen\footnotemark[1] ,
H.~K.~Kn\"orr\footnotemark[1]
}

\date{}

\maketitle

\begin{abstract}
We study the resolvent norm of a certain class of closed linear operators on a Hilbert space, including unbounded operators with compact resolvent.
It is shown that for any point in the resolvent set there exist directions in which the norm grows at least quadratically with the distance from this point.
This provides a new proof not using the maximum principle that the resolvent norm of the considered class cannot have local maxima. 
Finally, we give new criteria for the existence of local non-degenerate minima of the resolvent norm and provide examples of (un)bounded non-normal operators having this property.
\end{abstract}

\KEY
local growth estimate, extrema,
pseudospectra,
Schur complement
\EndKEY

\MSC
30D15, 
47A10, 
15A60 
\EndMSC

\section{Introduction}\label{sec: intro}

Let $\cH$ be a separable Hilbert space and let $A$ be a densely defined closed operator on $\cH$. Let $\rho(A)$ denote the resolvent set of $A$. Assume that $\rho(A)\neq\emptyset$. 
For $z\in \rho(A)$ the resolvent is denoted by $R_A(z)=(A-zI)^{-1}$. If for some $z\in\rho(A)$ the operator $S(z)=R_A(z)^{\ast}R_A(z)$ satisfies the spectral gap condition stated below we obtain a local growth estimate for the resolvent norm $\norm{R_A(z)}$. We apply the estimate to the question of (non-)existence of local extrema in the resolvent norm and the related question whether the level sets of $\norm{R_A(\cdot)}$ can have interior points. We give sufficient criteria for a local minimum in the resolvent norm and give a number of examples satisfying these criteria.

\begin{assumption}\label{assum}
Let $z\in\rho(A)$ be given. Assume that there exist $a(z)>0$ and $\lambda_{\rm max}(z)>a(z)$ such that $\sigma(S(z))\subseteq[0,a(z)]\cup\{\lambda_{\rm max}(z)\}$ and such that $\lambda_{\rm max}(z)$ is an eigenvalue of $S(z)$.
\end{assumption}

Note that this is an assumption on a single point $z \in \rho(A)$.

Assumption~\ref{assum} is satisfied for all $z\in\rho(A)$ in at least two generic cases. The first case is when $R_A(z_0)$ is compact for some $z_0\in \rho(A)$. Then $S(z)$ is compact and self-adjoint for all $z\in\rho(A)$. 
The second case is an operator of the form $A=\alpha I+K$ where $\alpha\in\IC$ and $K$ is a compact operator. If $\dim\cH = \infty$, $R_A(z)$ is never compact but $S(z)-|\alpha-z|^{-2} I$ is compact and self-adjoint for all $z\in\rho(A)$. In both cases, the norm of $S(z)$ is a discrete positive eigenvalue, but our results also hold true when the norm of $S(z)$ is an infinitely degenerate eigenvalue. In subsection \ref{ex-assum} we introduce a class of operators for which Assumption \ref{assum} is only satisfied on a subset of $\rho(A)$.
 
For $z,z'\in\IC$ we denote by $[z,z']$ the line segment from $z$ to $z'$.
\begin{theorem}\label{thm-horia}
Let $z\in\rho(A)$ be a point for which Assumption~{\rm\ref{assum}} holds.
Then there exist a constant $C>0$ and a point $z'\in \rho(A)$, $z'\neq z$, such that $[z,z']\subset\rho(A)$ and such that for every point $\zeta\in[z,z']$  we have
\begin{equation}\label{eq-thm1(i)}
\norm{R_A(\zeta)} \geq \norm{R_A(z)} + C \abs{\zeta-z}^2.
\end{equation}
\end{theorem}

It is well-known that the resolvent norm is subharmonic on $\rho(A)$, see e.g.~\cite[Theorem~4.2]{TE}, which by the maximum principle implies that it cannot have a  local maximum unless it is constant in an open set. 

The theorem implies that for $z\in \rho(A)$ satisfying Assumption~\ref{assum} the resolvent norm $\norm{R_A(\cdot)}$ cannot have a local maximum at $z$. Hence a level set of $\norm{R_A(\cdot)}$ cannot have $z$ as an interior point. We emphasize that the result is local since Assumption~\ref{assum} may be satisfied only in a subset of the resolvent set, see the example given in subsection~\ref{ex-assum}.

If $A$ is an unbounded closed operator with compact resolvent on a Hilbert space (or more generally a complex strictly convex Banach space) then it was proved recently in \cite[Theorem~2.2]{DS} that the resolvent level sets cannot have interior points.
An example of a closed unbounded operator on a Hilbert space with resolvent norm constant in a neighborhood of the origin was given in~\cite[Theorem ~3.2]{S2}.
Since the proof of Theorem~\ref{thm-horia} is based on $\norm{R_A(z)}^2=\norm{S(z)}=\lambda_{\rm max}(z)$, the Schur complement and perturbation theory, the non-existence of local maxima of the resolvent norm can be shown without using a maximum principle.

As a consequence of the proof we have the following two results. 
\begin{corollary}\label{cor12}
Assume that there exists $z\in\rho(A)$ satisfying Assumption~{\rm\ref{assum}} with the following properties:
\begin{itemize}
\item[\rm(i)] $\ip{\psi}{R_A(z) \psi} = 0$ for every eigenvector $\psi$ corresponding to the eigenvalue $\lambda_{\rm max}(z)$ of $R_A(z)^* R_A(z)$.
\item[\rm(ii)] $\ip{\psi}{R_A(z)^2 \psi} = 0$ for at least one of these eigenvectors.
\end{itemize}
Then the resolvent norm $\norm{R_A(\cdot)}$ has a local minimum at $z$.
\end{corollary} 
In subsections~\ref{sec: block diagonals}--\ref{infinite} we present examples of non-normal matrices and unbounded closed operators such that the resolvent norm has a local minimum.
\begin{corollary}\label{cor13}
Assume that there exists $z\in\rho(A)$ satisfying Assumption~{\rm\ref{assum}} with the following property:
\begin{itemize}
\item[\rm(i)] $\ip{\psi}{R_A(z) \psi} \neq 0$ for an eigenvector $\psi$ corresponding to the eigenvalue $\lambda_{\rm max}(z)$ of $R_A(z)^* R_A(z)$.
\end{itemize}
Then there exist a constant $C>0$ and a point $z'\in \rho(A),\ z'\neq z,$ such that $[z,z']\subset\rho(A)$ and such that for every point $\zeta\in[z,z']$  we have
\begin{equation}\label{cor1}
\norm{R_A(\zeta)} \geq \norm{R_A(z)} + C \abs{\zeta-z}.
\end{equation}
\end{corollary}

If $A$ is normal and has compact resolvent, then condition (i) in Corollary~\ref{cor13} holds for all $z\in\rho(A)$. We note that in this  case there is a simple direct proof of \eqref{cor1}, see Remark~\ref{normal-remark}. There may exist local minima in the case where the conditions in Corollary~\ref{cor12} are not satisfied. An example is the normal $3\times 3$-matrix $A=\diag(1,\e^{2\I\pi/3},\e^{-2\I\pi/3})$. Since $A$ is normal we have $\norm{R_A(z)}=1/\dist(z,\sigma(A))$. Then it is geometrically obvious that $z=0$ is a local minimum.

In the case $\dim \cH=2$ the resolvent norm cannot have any local extremum at all. 
This result is stated as follows:

\begin{theorem} \label{thm-henrik}
Let $A\in \IC^{2\times 2}$.
Then the resolvent norm $\norm{R_A(\cdot)}$ is symmetric with respect to $\tr(A)/2$ which is the average of the eigenvalues of $A$, i.e.
\begin{equation*}
\norm{R_A(\tr(A)/2+z)} = \norm{R_A(\tr(A)/2-z)}.
\end{equation*} 
Furthermore, 
\begin{enumerate}[\rm(i)]
\item If $A$ has an eigenvalue with algebraic multiplicity $2$, then $\norm{R_A(\cdot)}$ is a strictly decreasing radial function with center at the eigenvalue $\tr(A)/2$.
\item If $A$ has distinct eigenvalues then $\norm{R_A(\cdot)}$ has a saddle point at $\tr(A)/2$ and the following results hold.
\begin{enumerate}[\rm(a)]
\item If $A$ is normal, then the critical points consist of a line through $\tr(A)/2$ perpendicular to the line through the two eigenvalues of $A$. $\norm{R_A(\cdot)}$ is not real-differentiable on this line.
\item If $A$ is not normal then $\norm{R_A(\cdot)}$ is real-differentiable on $\rho(A)$ and $\tr(A)/2$ is the only critical point.
\end{enumerate} 
None of the critical points are local extrema.
\end{enumerate}
\end{theorem}

We can use the result in Theorem~\ref{thm-horia} to give a result on the pseudospectra. We recall the definition. For $\varepsilon>0$ the $\varepsilon$-pseudospectrum of $A$ is defined as
\begin{equation*}
\sigma_\varepsilon(A)=\{z\in \IC\,|\, \norm{R_A(z)} > \varepsilon^{-1}\}.
\end{equation*}
See~\cite{TE} for further information on pseudospectra.
We use the convention that $\norm{R_A(z)} = \infty$ if $(A - zI)$ is not invertible. We state the result in the finite dimensional case.
\begin{theorem}\label{path}
Let $A\in\IC^{N\times N}$. Then any point $z\in \sigma_\varepsilon(A)$ can be linked to one of the eigenvalues of $A$ through a finite polygonal path contained in $\sigma_\varepsilon(A)$.  
\end{theorem}

Note that this theorem implies the well-known result \cite[Theorem 2.4]{TE} that there must be at least one eigenvalue of $A$ in each connected component of $\sigma_{\varepsilon}(A)$.
Thus a pseudospectrum consists of at most $J$ connected components where $J$ is the number of distinct eigenvalues of $A$.
The novelty here is that we prove this fact by constructing a path inside a given pseudospectrum which connects any $z\in\sigma_{\varepsilon}(A)$ to an eigenvalue. This constructive approach is a consequence of the local growth estimate in \eqref{eq-thm1(i)}.

This article is organized as follows. The proofs of Theorems~\ref{thm-horia} and~\ref{thm-henrik} are given in sections~\ref{sec: proof horia} and~\ref{sec: proof henrik}. We construct the polygonal path for Theorem~\ref{path} in section~\ref{sec: connected components}. 
In section~\ref{sec: local minima} we provide the examples already mentioned above. In subsection~\ref{ex-assum} we give a class of operators for which Assumption~\ref{assum} only holds in a proper subset of the resolvent set. The examples in subsections~\ref{sec: block diagonals} and~\ref{sec: general example} are non-normal matrices satisfying the conditions in Corollary~\ref{cor12} such that the resolvent norm has a local minimum at the origin. In subsection~\ref{infinite} we finally give examples of unbounded closed operators with or without compact resolvent that satisfy Assumption~\ref{assum} such that the resolvent norm has a local minimum at the origin.

\section{Proof of Theorem~\ref{thm-horia}}%
\label{sec: proof horia}

Let $z\in\rho(A)$ such that Assumption~{\rm\ref{assum}} holds. This  $z$ is fixed throughout the proof. As above we let $S(z)=R_A(z)^{\ast}R_A(z)$. Note that $S(z)$ is a bounded, self-adjoint, and strictly positive operator.

For $\zeta\in\IC$ we introduce the notation $\Delta\zeta=\zeta-z$. Fix $\delta_1>0$ such that $\abs{\Delta\zeta}\leq\delta_1$ implies $\zeta\in\rho(A)$. Then a simple computation shows that we have
\begin{equation}\label{eq1}
\norm{S(\zeta)-S(z)}\leq C\abs{\Delta\zeta}, \quad \abs{\Delta\zeta}\leq\delta_1,
\end{equation}
where $C$ depends on $z$ and $\delta_1$.

Take $\delta_2=\frac12(\lambda_{\rm max}(z)-a(z))$ with $a(z)$ from Assumption~\ref{assum}. We can find $\delta_3>0$ such that for all $\abs{\Delta\zeta}\leq\delta_3$ and all $\lambda$ with $\abs{\lambda-\lambda_{\rm max}(z)}=\delta_2$ we have $\lambda\in\rho(S(\zeta))$ and
\begin{align}
(S(\zeta)-\lambda I)^{-1}&=(S(z)-\lambda I)^{-1}\bigl[
I+(S(\zeta)-S(z))(S(z)-\lambda I)^{-1}\bigr]^{-1}\notag\\
&=(S(z)-\lambda I)^{-1}\notag\\
&\quad-(S(z)-\lambda I)^{-1}(S(\zeta)-S(z))(S(z)-\lambda I)^{-1}\notag\\
&\qquad\cdot\bigl[
I+(S(\zeta)-S(z))(S(z)-\lambda I)^{-1}\bigr]^{-1}.\label{neumann}
\end{align}
We now use some standard arguments from perturbation theory, see~\cite{Kato}. Note that the map $z\mapsto S(z)$ is norm continuous but not analytic.
We define the Riesz projections
\begin{equation}
P(\zeta)=\frac{-1}{2\pi \I}\int_{\abs{\lambda-\lambda_{\rm max}(z)}=\delta_2} (S(\zeta)-\lambda I)^{-1}\di\lambda.
\end{equation}
We write $P=P(z)$, which is the eigenprojection of the eigenvalue $\lambda_{\rm max}(z)$. 
Using \eqref{neumann} we can find a $\delta_4$, $0<\delta_4\leq\delta_3$, such that for all $\abs{\Delta\zeta}\leq\delta_4$ we have $\norm{P(\zeta)-P}<1$. There exists a family of unitary operators $U(\zeta)\colon\ran P(\zeta)\to\ran P$ such that 
$U(\zeta)P(\zeta)=P U(\zeta)$, see~\cite[I-\S6.8]{Kato}. Note that this result holds in infinite dimensions. Together with the upper semi-continuity of the parts of the spectrum of $S(\zeta)$, see~\cite[IV-\S3.4]{Kato}, we conclude that 
$\norm{S(\zeta)}=\norm{S(\zeta)P(\zeta)}$. 
Let us define
\begin{equation}\label{defD}
D=\{\lambda\in \IC\,|\, \abs{\lambda-\lambda_{\rm max}(z)}<\delta_2\},
\end{equation}
such that $\sigma(S(\zeta)P(\zeta))=\sigma(S(\zeta))\cap D$.
We have
\begin{equation}
\dist(\lambda_{\rm max}(z),\sigma(S(\zeta)P(\zeta)))\leq C\abs{\Delta\zeta},\quad
\abs{\Delta\zeta}\leq\delta_4.
\end{equation}

Let $P^{\perp}=I-P$. Then for all $\lambda$ satisfying $\abs{\lambda-\lambda_{\rm max}(z)}\leq C\abs{\Delta\zeta}$ and all $\abs{\Delta\zeta}\leq\delta_4$ we see that $P^{\perp}(S(\zeta)-\lambda I)P^{\perp}$ is invertible in $\ran P^{\perp}$. We now use the Schur complement based on $P$ and $P^{\perp}$, also known as the Feshbach formula, see e.g.\ \cite[Equations~(6.1)-(6.2)]{Nen}.

We have that $\lambda\in\rho(S(\zeta))$ if and only if the Schur complement
\begin{equation}
F(\zeta,\lambda)=PS(\zeta)P-\lambda P+
PS(\zeta)P^{\perp}\bigl(\lambda P^{\perp}-P^{\perp} S(\zeta)P^{\perp}\bigr)^{-1}P^{\perp} S(\zeta)P
\end{equation}
is invertible in $\ran P$. In the affirmative case we have
\begin{equation}\label{inverseF}
F(\zeta,\lambda)^{-1}=P(S(\zeta)-\lambda I)^{-1}P.
\end{equation}

We note that with 
$\abs{\lambda-\lambda_{\rm max}(z)}\leq C \abs{\Delta\zeta}$
we have
\begin{equation}\label{perp-comp}
\norm{
\bigl(P^{\perp}(S(\zeta)-\lambda I)P^{\perp}\bigr)^{-1}-
\bigl(P^{\perp}(S(z)-\lambda_{\rm max}(z) I)P^{\perp}\bigr)^{-1}
}\leq C\abs{\Delta\zeta},
\end{equation}
where the norm is the operator norm on $\ran P^{\perp}$.

\smallskip

Let $M$ and $N$ be two non-empty compact sets in $\IC$.
The Hausdorff distance between $M$ and $N$ is defined as 
\begin{equation*}
d_{\rm H}(M,N) = \max\bigl\{\sup_{\mu\in N} \dist(\mu,M),\; \sup_{\mu\in M} \dist(\mu,N)\bigr\}.
\end{equation*}
The geometric interpretation is that given any $\mu\in M$ we can find at least one 
$\nu\in N$ such that $\abs{\mu-\nu} \leq d_{\rm H}(M,N)$, and vice versa. 

\begin{lemma}\label{lemma1}
Let $\zeta \in \rho(A)$ and write $\Delta S(\zeta) = S(\zeta) - S(z)$.
Define the operator
\begin{align*}
W(\zeta)&= \lambda_{\rm max}(z) P + P\Delta S(\zeta)P\\
&\quad
+P\Delta S(\zeta)P^{\perp} \bigl(\lambda_{\rm max}(z)P^{\perp}-P^{\perp} S(z)P^{\perp} \bigr)^{-1} P^{\perp} \Delta S(\zeta)P \end{align*}
on  $\ran P$.
Then the spectrum of $W(\zeta)$ is at a Hausdorff distance of order $\abs{\Delta\zeta}^3$ from $\sigma(S(\zeta))\cap D$.
\end{lemma}

\begin{proof}
Let $\zeta \in \rho(A)$. 
In the first part of the proof we will show that if $\lambda\in D$ is located at a distance larger than some constant times $\abs{\Delta\zeta}^3$ from the spectrum of $S(\zeta)$ then it must belong to the resolvent set of $W(\zeta)$.
In other words, no points of $D$ belonging to the spectrum of $W(\zeta)$ can be at a distance larger than $C \abs{\Delta\zeta}^3$ from the spectrum of $S(\zeta)$. 

Indeed, let us assume that $\dist(\lambda,\sigma(S(\zeta)))>0$.  
Due to \eqref{inverseF} and the self-adjointness of $S(\zeta)$ we have 
\begin{equation*}
\norm{F(\zeta,\lambda)^{-1}} \leq \frac{1}{\dist(\lambda,\sigma(S(\zeta)))}.
\end{equation*}
Note that since {$P^{\perp} S(\zeta)P= P^{\perp} \Delta S(\zeta)P$ 
and $P S(\zeta)P^\perp=P \Delta S(\zeta)P^\perp$,}
the norms of the off-diagonal components $P^\perp S(\zeta)P$ and $P S(\zeta)P^\perp$ are of order $\abs{\Delta \zeta}$ by \eqref{eq1}, which combined with \eqref{perp-comp} gives the estimate
\begin{equation*}
\norm{F(\zeta,\lambda)-(W(\zeta)-\lambda P)} \leq C \abs{\Delta \zeta}^3.
\end{equation*}
Hence
\begin{equation*}
W(\zeta)-\lambda P = \bigl( I -[F(\zeta,\lambda)-(W(\zeta)-\lambda P)]F(\zeta,\lambda)^{-1}\bigr)F(\zeta,\lambda)
\end{equation*}
is invertible in $\ran (P)$ if 
\begin{equation*}
\frac{C \abs{\Delta \zeta}^3}{\dist(\lambda,\sigma(S(\zeta)))}<1.
\end{equation*}
Thus if {$\lambda\in D$} and $\dist(\lambda,\sigma(S(\zeta))) > C \abs{\Delta \zeta}^3$, then $\lambda$ is not in the spectrum of $W(\zeta)$ and {the first part of the proof is finished}.

Now we prove the second part, i.e.\ we show that any point $\lambda\in D$ which is located at a distance larger than $C \abs{\Delta \zeta}^3$ from the spectrum of $W(\zeta)$ must  belong to the resolvent set of $S(\zeta)$. 

Indeed, let us assume that $\dist(\lambda,\sigma(W(\zeta)))>0$.
Then
\begin{equation*}
F(\zeta,\lambda)=\bigl(I +[F(\zeta,\lambda)-(W(\zeta)-\lambda P)](W(\zeta)-\lambda P)^{-1}\bigr)(W(\zeta)- \lambda P)
\end{equation*}
is invertible if 
\begin{equation*}
\frac{C \abs{\Delta \zeta}^3}{\dist(\lambda,\sigma(W(\zeta)))}<1.
\end{equation*}
Here we used that $W(\zeta)$ is self-adjoint.
We conclude that $F(\zeta,\lambda)$, and therefore $S(\zeta)-\lambda I$, is invertible for {such $\lambda$'s, hence no element of $D$ which belongs to the spectrum of $S(\zeta)$ can be located at a distance larger than  $C \abs{\Delta \zeta}^3$ from the spectrum of $W(\zeta)$}.
\end{proof}

The following proposition is a direct consequence of Lemma~\ref{lemma1}. We abuse notation slightly and write $\lambda_{\rm max}(\zeta)=\norm{S(\zeta)}$. We emphasize that in the case where $\lambda_{\rm max}(z)$ has infinite multiplicity $\lambda_{\rm max}(\zeta)$ need not be an eigenvalue of $S(\zeta)$.

\begin{proposition}\label{lemma2}
We have that
\begin{equation}\label{hc54}
\bigl\lvert\lambda_{\rm max}(\zeta)-\norm{W(\zeta)}\bigr\rvert \leq C \abs{\Delta \zeta}^3.
\end{equation}
\end{proposition}

\begin{proof} 
If $\abs{\Delta\zeta}$ is small enough, then both $\norm{W(\zeta)}$ and $\lambda_{\rm max}(\zeta)$ belong to $D$. 

Assume without loss of generality that $\lambda_{\rm max}(\zeta) > \norm{W(\zeta)}$.
Since $\norm{W(\zeta)}$ is the element of $\sigma(W(\zeta))$ which is closest to $\lambda_{\rm max}(\zeta)$ we have 
\begin{equation*}
0 < \lambda_{\rm max}(\zeta)-\norm{W(\zeta)} \leq d_{\rm H}\big (\sigma(W(\zeta)){\cap D},\sigma(S(\zeta)){\cap D}\big ) \leq C \abs{\Delta \zeta}^3.
\end{equation*} 
\end{proof}

\begin{proposition}\label{propo9}
There exist $C>0$ and a point $z'\in\rho(A)\setminus\{z\}$ such that $[z,z']\subset\rho(A)$ and $\lambda_{\rm max}(\zeta)-\lambda_{\rm max}(z) \geq C\abs{\zeta-z}^{2}$ for every $\zeta \in [z,z']$. 
\end{proposition}
\begin{remark}
Recall that $\norm{R_A(z)}^2 = \lambda_{\rm max}(z)$ and $\norm{R_A(\zeta)}^2 = \lambda_{\rm max}(\zeta)$. 
If {$\abs{z'-z}$} is small enough we have $\norm{R_A(\zeta)} \leq 2 \norm{R_A(z)}$, 
$\zeta\in[z,z']$.
This implies for $\zeta\in[z,z']$ 
\begin{equation*}
\norm{R_A(\zeta)} - \norm{R_A(z)} = \frac{\norm{R_A(\zeta)}^2 
-\norm{R_A(z)}^2}{\norm{R_A(\zeta)} + \norm{R_A(z)}} \geq 
\frac{\lambda_{\rm max}(\zeta)-\lambda_{\rm max}(z)}{3 \norm{R_A(z)}} \geq C \abs{z-\zeta}^2.
\end{equation*}
Thus this proposition implies Theorem~\ref{thm-horia}. 
\end{remark}
\begin{proof}
Use the Taylor expansion to get
\begin{equation*}
R_A(\zeta) = R_A(z) + (\Delta\zeta)R_A(z)^2 + (\Delta\zeta)^2R_A(z)^3 + \mathcal{O}(\abs{\Delta \zeta}^3).
\end{equation*}
Let $\Delta S(\zeta)$ and $W(\zeta)$ be defined as in Lemma~\ref{lemma1}.
Then we get
\begin{align}\label{hc10}
\Delta S(\zeta) &= (\Delta \zeta) S(z) R_A(z) + (\overline{\Delta\zeta}) R_A(z)^* S(z)\nonumber \\
 &\quad +(\Delta\zeta)^2 S(z) R_A(z)^2 + (\overline{\Delta\zeta})^2 (R_A(z)^*)^2 S(z)\nonumber \\
 &\quad +\abs{\Delta \zeta}^2 R_A(z)^* S(z) R_A(z) + \mathcal{O}(\abs{\Delta \zeta}^3).
\end{align}
Next we introduce the operator
\begin{align}\label{hc12}
\widetilde{W}(\zeta) &= \lambda_{\rm max}(z) P + \lambda_{\rm max}(z) (\Delta\zeta) P R_A(z) P + \lambda_{\rm max}(z) (\overline{\Delta\zeta}) P R_A(z)^* P \nonumber \\
&\phantom{=.} + \lambda_{\rm max}(z) (\Delta\zeta)^2 PR_A(z)^2 P + \lambda_{\rm max}(z) (\overline{\Delta\zeta})^2 P (R_A(z)^*)^2 P\nonumber \\
&\phantom{=.} + \abs{\Delta\zeta}^2 P R_A(z)^* S(z) R_A(z) P \nonumber \\
&\phantom{=.} + P \Delta S(\zeta) P^\perp \bigl(\lambda_{\rm max}(z) P^\perp - P^\perp S(z) P^\perp \bigr)^{-1} P^\perp \Delta S(\zeta) P.
\end{align}
The estimate 
\begin{equation*}
\norm{W(\zeta)-\widetilde{W}(\zeta)} \leq C \abs{\Delta\zeta}^3
\end{equation*}
yields that the Hausdorff distance between the spectra of $W(\zeta)$ and $\widetilde{W}(\zeta)$ is of order $\abs{\Delta\zeta}^3$.
Using Proposition~\ref{lemma2} we conclude that $\lambda_{\rm max}(\zeta)$ is at a distance of order $\abs{\Delta \zeta}^3$ from $\norm{\widetilde{W}(\zeta)}$. 

The main idea is now to find $z' \in \rho(A)$ such that for every $\zeta \in [z,z']$,
\begin{equation*}
\norm{\widetilde{W}(\zeta)} \geq \lambda_{\rm max}(z) + C_1 \abs{\Delta \zeta} + C_2 \abs{\Delta \zeta}^2
\end{equation*}
where $C_1, C_2\geq0$ with $\max\{C_1,C_2\}>0$.
This will prove that for $\abs{z-z'}$ small enough, either $(\lambda_{\rm max}(\zeta)-\lambda_{\rm max}(z))/\abs{\Delta\zeta}$ or $(\lambda_{\rm max}(\zeta)-\lambda_{\rm max}(z))/\abs{\Delta\zeta}^{2}$ is bounded from below by a positive constant on $[z,z']$.

First assume that there exists $\psi=P\psi$ with $\norm{\psi}=1$, such that 
\begin{equation*}
\ip{\psi}{R_A(z) \psi} = \eta \e^{\I\varphi},\quad \eta>0,\quad \varphi\in [0,2\pi[.
\end{equation*}
Now choose $\zeta\in [z,z']$ with $z'=z+r\e^{-\I\varphi}$.
Then 
\begin{equation}\label{order1}
\norm{\widetilde{W}(\zeta)} \geq \ip{\psi}{\widetilde{W}(\zeta)\psi} \geq \lambda_{\rm max}(z) + 2\lambda_{\rm max}(z) \abs{\Delta\zeta} \eta + \mathcal{O}(\abs{\Delta\zeta}^2).
\end{equation}
In this case we can choose $C_1 = 2 \lambda_{\rm max}(z) \eta$, $C_2=0$, and we are done. 

Next assume that $\ip{\psi}{R_A(z)\psi} =0$ for all $\psi=P\psi$ of norm one.
Using 
\begin{equation*}
\ip{\Delta S(\zeta)\psi}{P^\perp \bigl(\lambda_{\rm max}(z) P^\perp - P^\perp S(z)P^\perp \bigr)^{-1} P^\perp \Delta S(\zeta)\psi} \geq 0
\end{equation*}
in \eqref{hc12} we obtain
\begin{align}\label{hc13}
\norm{\widetilde{W}(\zeta)} &\geq \ip{\psi}{\widetilde{W}(\zeta)\psi}\nonumber \\
&\geq \lambda_{\rm max}(z)+2\lambda_{\rm max}(z) \re \bigl( (\Delta\zeta)^2 \ip{\psi}{R_A(z)^2 \psi} \bigr)\nonumber \\
&\quad + \abs{\Delta \zeta}^2 \ip{R_A(z) \psi}{S(z) R_A(z)\psi}.
\end{align}

Let $C_2=\ip{R_A(z) \psi}{S(z)R_A(z)\psi}$. Since $S(z)$ is strictly positive and $R_A(z)\psi\neq0$, we get $C_2>0$.

This result inserted into \eqref{hc13} gives
\begin{equation*} 
\norm{\widetilde{W}(\zeta)} \geq \lambda_{\rm max}(z)+2\lambda_{\rm max}(z) \re \bigl((\Delta\zeta)^2\ip{\psi}{R_A(z)^2\psi} \bigr) + C_2\abs{\Delta \zeta}^2.
\end{equation*}
Write 
\begin{equation*}
\ip{\psi}{R_A(z)^2 \psi} = \eta \e^{\I\varphi}, \quad \eta\geq 0,\quad \varphi\in [0,2\pi[.
\end{equation*}
Then choosing $\zeta\in[z,z']$ with $z'=z+r\e^{-\I\varphi/2}$ we get
$\re \bigl((\Delta\zeta)^2\ip{\psi}{R_A(z)^2\psi} \bigr)\geq0$ and then
\begin{equation}\label{estimate}
\norm{\widetilde{W}(\zeta)} \geq \lambda_{\rm max}(z) + C_2\abs{\Delta\zeta}^2.
\end{equation}

It remains to consider the case when $\cH$ is finite dimensional and $\lambda_{\max}(z)$ is the only eigenvalue of $S(z)$. In this case $P=I$ and $P^{\perp}=0$. If one omits all terms involving $P^{\perp}$ in the proof above then it is valid also in this case.

This concludes the proof of Theorem~\ref{thm-horia}. 
\end{proof}

Corollary~\ref{cor13} is an immediate consequence of the proof above. If the conditions in Corollary~\ref{cor12} are satisfied
and we take $\psi$ satisfying (ii), then we get the estimate \eqref{estimate} for any choice of direction, such that $z$ is a local minimum point of the resolvent norm.

\begin{remark}\label{normal-remark}
We note that in the case $A$ normal with compact resolvent there is a simple direct proof of the estimate~\eqref{cor1} in Corollary~\ref{cor13}. Let $z\in\rho(A)$. Since $\sigma(A)$ is discrete, there exists $\lambda\in\sigma(A)$ such that $\dist(z,\sigma(A))=\abs{\lambda-z}$ and therefore $\{z+t(\lambda-z) \,|\, t\in[0,1[\,\}\subset\rho(A)$. Take $z'\neq z$ in this set, sufficiently close to $z$. Since $A$ is normal, we have
\begin{equation}
\norm{R_A(z)}=\frac{1}{\dist(z,\sigma(A))}.
\end{equation}
Thus the result follows by a simple geometrical argument.

We can also prove the estimate~\eqref{cor1} for $A$ normal with compact resolvent by verifying the condition (i) in Corollary~\ref{cor13}. We have that $R_A(z)$ is normal for all $z\in\rho(A)$. Fix $z\in\rho(A)$. Assume that $\ip{\psi}{R_A(z)\psi}=0$ for all $\psi\in\ran P$. By polarization we get that $\ip{\psi'}{R_A(z)\psi}=0$ for all $\psi',\psi\in\ran P$. Since $R_A(z)$ and $R_A(z)^{\ast}$ commute we get that $R_A(z)\colon\ran P\to\ran P$ is an isomorphism. Take $\psi'=R_A(z)\psi$. Then it follows from $\ip{\psi'}{R_A(z)\psi}=0$ that $\psi=0$, a contradiction. Thus there exists $\psi\in\ran P$ with 
$\ip{\psi}{R_A(z)\psi}\neq0$.
\end{remark}

\section{Proof of Theorem~\ref{thm-henrik}}%
\label{sec: proof henrik}

Let
\begin{equation*}
A = \begin{bmatrix}
	a & b \\ 
        c & d
    \end{bmatrix} \in \IC^{2\times 2}.
\end{equation*}
Let $z\in\rho(A)$ and set $T(z) = (A-zI)^*(A-zI)$. Define
\begin{align*}
	w(z) &= \tr(T(z)) = \sum_{i,j=1}^2\abs{(A-zI)_{ij}}^2 = \abs{a-z}^2+\abs{d-z}^2+\abs{b}^2 + \abs{c}^2, \\
	h(z) &= \det(T(z)) = \abs{\det(A-zI)}^2 = \abs{z^2-\tr(A)z+\det(A)}^2.
\end{align*}
The resolvent norm $\norm{R_A(z)}$ equals the reciprocal to the smallest singular value $s(z)$ of $A-zI$, i.e.\ $s(z)^2$ is the smallest eigenvalue of the positive definite matrix $T(z)$. 
This leads to the expression
\begin{equation*}
\norm{R_A(z)}^2 = \frac{1}{s(z)^2} = \frac{2}{w(z)-\sqrt{w(z)^2-4h(z)}}.
\end{equation*}
The average of the eigenvalues of $A$ is $\tr(A)/2$, thus
\begin{equation*}
A_1 = A - \frac{\tr(A)}{2}I 
\end{equation*}
has eigenvalues $\pm \lambda$ for some $\lambda\in \IC$. 
We consider the two cases (i) $\lambda = 0$ and (ii) $\lambda \neq 0$ separately.

\paragraph{Case (i)} 
We start with the case $\lambda=0$. Then $\tr(A_1) = \det(A_1) = 0$ which implies that either 
\begin{equation}
A_1 = \begin{bmatrix}
       0 & 0 \\ c & 0
      \end{bmatrix} 
\qquad\text{or}\qquad 
A_1 = \begin{bmatrix}
	a & b \\ -\frac{a^2}{b} & -a
      \end{bmatrix} 
\label{eq:type1}
\end{equation}
with $a,b,c\in\IC$ and $b \neq 0$.
The corresponding expressions for $w$ and $h$ for the matrices in \eqref{eq:type1} are $w(z) = 2t+k$ and $h(z) = t^2$ where $t = \abs{z}^2$ and $k = \abs{c}^2$, respectively $k = 2\abs{a}^2+\abs{b}^2+\frac{\abs{a}^4}{\abs{b}^2}$. 
Thus in both cases $\norm{R_{A_1}(\cdot)}$ is a radial function
\begin{equation*}
\norm{R_{A_1}(z)}^2 = \frac{2}{2t+k-(4kt+k^2)^{1/2}}.
\end{equation*}
For $k = 0$ we have $\norm{R_{A_1}(z)} = \abs{z}^{-1}$, so we may assume $k>0$. The derivative is
\begin{equation*}
\frac{d}{dt}\norm{R_{A_1}(z)}^2 = \frac{-4+4k(4kt+k^2)^{-1/2}}{(2t+k-(4kt+k^2)^{1/2})^2}.
\end{equation*}
As the eigenvalues of $A_1$ are 0 then $t>0$ and $k>0$, i.e.\ we have $\frac{d}{dt}\norm{R_{A_1}}^2 < 0$.
Since the resolvent for $A$ is given by
\begin{equation}
R_A(z) = R_{A_1}(z-\tfrac{1}{2}\tr(A)) \label{eq:resnormtype1}
\end{equation}
we conclude that $\norm{R_A(\cdot)}$ is a radial function with center at its eigenvalue $\tr(A)/2$ and is strictly decreasing away from $\tr(A)/2$.

\paragraph{Case (ii)} Assume $\lambda\neq 0$. As the eigenvalues of $A_1$ are distinct
\begin{equation*}
A_2 = \frac{1}{\lambda}A_1
\end{equation*}
has eigenvalues $\pm 1$. Thus $\tr(A_2) = 0$ and $\det(A_2) = -1$ whence either  
\begin{equation}
A_2 = \pm\begin{bmatrix}
	1 & 0 \\ c & -1
         \end{bmatrix}
\qquad\text{or}\qquad 
A_2 = \begin{bmatrix}
	a & b \\ \frac{1-a^2}{b} & -a
      \end{bmatrix} 
\label{eq:type2}
\end{equation}
with $a,b,c\in\IC$ and $b \neq 0$.
The corresponding expressions for $w$ and $h$ are $w(z) = 2\abs{z}^2 + k$ and $h(z) = \abs{z}^4+1-z^2-\overline{z}^2$ where $k = 2+\abs{c}^2$ and $k = 2\abs{a}^2+\abs{b}^2+\frac{\abs{1-a^2}^2}{\abs{b}^2}$, respectively. 
Note that since $w^2-4h\geq 0$ a priori, $k\geq 2$. In fact $k = 2$ if and only if $A_2$ is self-adjoint, {which is the case if and only if $A$ is normal.} 
For $k = 2$ we have for $z = x_1 + \I x_2$ with $x_1,x_2\in\IR$,
\begin{equation*}
\norm{R_{A_2}(z)}^2 = \frac{1}{(1-\abs{x_1})^2+x_2^2}.
\end{equation*}
This function is not differentiable at $x_1=0$, but clearly it increases away from the imaginary axis for each fixed $x_2$ between $z = \pm 1+\I x_2$, and decreases away from the origin on the imaginary axis. 
It is easily checked that there are no critical points for $\abs{x_1}>0$ as $\pm 1\in\sigma(A_2)$. 
In particular, $z=0$ is a saddle point of $\norm{R_{A_2}(\cdot)}$ and there are no local extrema.

Now assume $k>2$.
Then
\begin{align}
\norm{R_{A_2}(z)}^2 &= \frac{2}{2\abs{z}^2+k-((k+2)(k-2)+4(k\abs{z}^2+z^2+\overline{z}^2))^{1/2}} \notag\\
	       &= \frac{2}{2x_1^2+2x_2^2+k-((k+2)(k-2)+4(k+2)x_1^2+4(k-2)x_2^2)^{1/2}}. \label{eq:resnormA2}
\end{align}
It is evident that $\norm{R_{A_2}(\cdot)}$ is symmetric with respect to the origin. 

By a straightforward direct calculation
\begin{align*}
\frac{\partial}{\partial x_1}\norm{R_{A_2}(z)}^2 &= 2x_1\norm{R_{A_2}(z)}^4 ((k+2)(w(z)^2-4h(z))^{-1/2}-1), \\
\frac{\partial}{\partial x_2}\norm{R_{A_2}(z)}^2&= 2x_2\norm{R_{A_2}(z)}^4 ((k-2)(w(z)^2-4h(z))^{-1/2}-1).
\end{align*}
Note that as $k>2$ then $w^2-4h\geq (k+2)(k-2) > (k-2)^2$.
Thus we must have $x_2 = 0$ at a critical point. 
If $x_2 = 0$, then $w^2-4h = (k+2)^2$ if and only if $x_1 = \pm 1$.
However, as $z = \pm 1 \in \sigma(A_2)$ the only critical point of $\norm{R_{A_2}(\cdot)}$ is at $z = 0$.

From \eqref{eq:resnormA2} and writing $z = \sqrt{t}\e^{\I\theta}$, we get
\begin{equation*}
	\norm{R_{A_2}(z)}^2 = \frac{2}{2t+k-((k+2)(k-2)+4t(k+2\cos(2\theta)))^{1/2}}.
\end{equation*}
From this we immediately obtain
\begin{align*}
\frac{d}{dt}\norm{R_{A_2}(z)}^2 &= \norm{R_{A_2}(z)}^4 ((k+2\cos(2\theta))(w(z)^2-4h(z))^{-1/2}-1) \\
	                   &= \frac{4\norm{R_{A_2}(z)}^4 (k+2\cos(2\theta))
(w(z)^2-4h(z))^{-1/2}}{k+2\cos(2\theta)+(w(z)^2-4h(z))^{1/2}}(g(k,\theta)-t)
\end{align*}
where
\begin{equation*}
g(k,\theta) = \frac{1}{4}(k+2\cos(2\theta)) - \frac{(k+2)(k-2)}{4(k+2\cos(2\theta))}.
\end{equation*}
For $k>2$ all the terms in $\frac{d}{dt}\norm{R_{A_2}(z)}^2$ are positive except $g(k,\theta)-t$. 
As a consequence, $g(k,\theta)$ determines the sign of the derivative
\begin{equation}
\begin{cases}
\dfrac{d}{dt}\norm{R_{A_2}(z)}^2 > 0, & 0<t<g(k,\theta), \\[10pt]
\dfrac{d}{dt}\norm{R_{A_2}(z)}^2 < 0, & t > \max\{g(k,\theta),0\}.
\end{cases} \label{eq:deriv}
\end{equation}
Since $k > 2$ we have in particular $\frac{d}{dt}\norm{R_{A_2}(x_1)}^2>0$ for $0< t < 1$ and $\frac{d}{dt}\norm{R_{A_2}(\I x_2)}^2<0$ for $t>0$. 
Thus $\norm{R_{A_2}(\cdot)}$ is increasing away from the origin on the real axis between $z=\pm 1$ and decreasing away from the origin on the imaginary axis.

The resolvent for $A$ is now given by
\begin{equation}
{R_A(z) = \lambda^{-1}R_{A_2}\bigl( \lambda^{-1}(z-\tfrac{1}{2}\tr(A))\bigr).} \label{eq:resnormtype2}
\end{equation}
This implies that for $k>2$, $\norm{R_A(\cdot)}$ has exactly one critical point at $\tr(A)/2$ which is a saddle point, and is furthermore symmetric with respect to the point $\tr(A)/2$. 
{For $k=2$ (when $A$ is normal) we also have a line of critical points where $\norm{R_A(\cdot)}$ is non-differentiable, as stated in the theorem, none of which are local extrema.}


\section{Proof of Theorem~\ref{path}}%
\label{sec: connected components}

Assume that $A\in \IC^{N\times N}$ has the distinct eigenvalues $\lambda_1,\ldots, \lambda_J$, $1\leq J\leq N$. Let $\psi_j$ denote a normalized eigenvector corresponding to the eigenvalue $\lambda_j$, $j=1,\ldots,J$.
Then we have for any $z \in \rho(A)$ and $j=1,\ldots,J$,
$\norm{R_A(z)} \geq \norm{R_A(z) \psi_j} = 1/{\abs{\lambda_j-z}}$.
Thus
\begin{equation*}
\norm{R_A(z)} \geq \frac{1}{\dist(z,\sigma(A))}.
\end{equation*}
For $\epsilon>0$ this estimate implies
$B_{\varepsilon/2}(\lambda_j) \subset \sigma_\varepsilon(A)$ for each $j=1,\ldots,J$.

We have that $\sigma_{\varepsilon}(A)$ is bounded, 
since $\norm{R_A(z)}\to0$ as $\abs{z}\to \infty$.
Thus $\overline{\sigma_\varepsilon(A)}$ is compact.
If $z \in \sigma(A)$ then $\norm{R_A(z)}$ is interpreted to be $\infty$. Actually we have $\overline{\sigma_\varepsilon(A)}=\{z\in\IC \,|\, \norm{R_A(z)} \geq 1/{\varepsilon}\}$; see e.g.\ \cite{CCH,DS,H,S} for this result and for results on $\{z\in\IC \,|\, \norm{R_A(z)} = 1/{\varepsilon}\}$.

Fix $\varepsilon>0$. Define $f\colon \overline{\sigma_\varepsilon(A)} \to [\varepsilon^{-1},\infty]$ by $f(z) = \norm{R_A(z)}$.
Consider the compact set 
\begin{equation*}
K = \overline{\sigma_\varepsilon(A)} \setminus \bigl(\bigcup_{j=1}^J B_{\varepsilon/2}(\lambda_j) \bigr).
\end{equation*}
$f$ is bounded on $K$ and $\dist(z,\sigma(A)) \geq \varepsilon/2$ for any $z \in K$. 

Next fix $z \in \sigma_{\varepsilon}(A)$.
The remainder of the proof constructs a sequence $(x_n)_{n\in\IN}\subset \sigma_\varepsilon(A)$ which starts at the point $x_1=z$ and at a finite index $m$ enters $B_{\varepsilon/2}(\lambda_j)$ for some $j=1,\dots,J$.
Moreover, each line segment $[x_n,x_{n+1}]$, $1 \leq n \leq m-1$, lies completely inside $\sigma_\varepsilon(A)$.

{If $z\in\sigma_\varepsilon(A)\setminus K$ we are done, as $z$ can be connected to an eigenvalue by a single line segment. Hence we may assume $z$ belongs to the interior of $K$.} Set $\delta = (f(z)-\varepsilon^{-1})/2>0$ and for any $x\in \overline{\sigma_{\varepsilon}(A)}\setminus\sigma(A)$ define $r_x$ to be the supremum over all $r>0$ such that there exists a $y$ with $\abs{x-y}=r$, such that {$f(x)<f(y)$} and $f(x)-\delta\leq f(\zeta)$ for all $\zeta\in[x,y]$.

Note that this implies $y \in \sigma_{\varepsilon}(A)$, but not necessarily $[x,y] \subset \overline{\sigma_{\varepsilon}(A)}$.
By Theorem~\ref{thm-horia} the supremum is taken over a non-empty set,
such that $r_x > 0$.
As $\overline{\sigma_{\varepsilon}(A)}$ is bounded the point $y$ cannot be located arbitrarily far away from $x$ so $r_x$ is finite.

For each $x\in \overline{\sigma_{\varepsilon}(A)}\setminus\sigma(A)$ there exists $y_x \in \sigma_{\varepsilon}(A)$ such that $r_x/2 < \abs{y_x - x} \leq r_x$, $f(x)<f(y_x)$, and $f(x) - \delta \leq f(\zeta)$ for all $\zeta \in [x,y_x]$. 

Define a sequence $(x_n)_{n \in \IN} \subset \sigma_\varepsilon(A)$ by $x_1=z$ and $x_{n+1}=y_{x_n}$, $n\geq 1$, which implies 
\begin{equation*}
\frac{1}{\varepsilon} < f(z)=f(x_1)<f(x_2)<\dots
\end{equation*}
Next we show that $x_n$ must escape from $K$ for some $n$.
Assume that this is not true such that $x_n\in K$ for all $n$.
Then the increasing sequence given by $f(x_n)$ is bounded. Thus there exists $M>0$ such that $f(x_n)<M$ for all $n$ and $\lim_{n\to\infty}f(x_n)=M$.
As $K$ is compact we may assume (by passing to a subsequence) that $x_n$ converges to $a\in K$. 

From 
Theorem~\ref{thm-horia} we know that there exists some $a'$ located at a positive distance from $a$ where $f(a)<f(a')$ and $f(a)<f(\zeta)$ for all $\zeta$ in the open line segment $]a,a'[$. 

Since $f$ is uniformly continuous on $K$, we have for $n$ larger than some $m$, $f(x_n)<f(a')$ and $f(x_n)-\delta\leq f(\zeta_n)$ for every $\zeta_n\in[x_n,a']$.
Furthermore, we may assume that $\abs{x_{n+1}-x_n} \leq \abs{a-a'}/10$ and $\abs{a-x_{n+1}}\leq \abs{a-a'}/10$.
Then $a'$ fulfils the criteria for $y$ which were used to define $r_{x_n}$.  If $n$ is large enough we have 
\begin{equation*}
\abs{a'-x_n} \geq \abs{a'-a}-\abs{a-x_{n+1}} - \abs{x_{n+1}-x_n} > 2\abs{x_{n+1}-x_n} = 2 \abs{y_{x_n}-x_n} > r_{x_n},
\end{equation*}
contradicting the definition of $r_{x_n}$. 
Thus the $x_n$ must lie outside $K$ if $n \geq m$ for some $m$, and they must lie in $\bigcup_{j=1}^J B_{\varepsilon/2}(\lambda_j)$.

It remains  to show that the polygonal path connecting $z$, via the points $z=x_1,\dots,x_m$, to $\bigcup_{j=1}^J B_{\varepsilon/2}(\lambda_j)$ is contained in $\sigma_{\varepsilon}(A)$. From the definition of $\delta$ and the construction of $x_1,\dots,x_m$, every $\zeta$ on the polygonal path satisfies
\begin{equation*} 
f(\zeta)\geq f(z)-\delta= \frac{f(z)+\varepsilon^{-1}}{2} > \frac{1}{\varepsilon}.
\end{equation*}
Thus the path lies in $\sigma_\varepsilon(A)$, and can with one additional line segment be connected to one of the eigenvalues of $A$. 

\begin{remark}
The connected components of $\sigma_{\varepsilon}(A)$ are not necessarily simply connected. In the case of $A$ normal an example is constructed as follows.
Let $N\geq2$ and let $A\in\IC^{N\times N}$ be the diagonal matrix with $A_{jj}=j+\I(-1)^j\sqrt{3}/2$, $j=1,2,\ldots,N$. Elementary geometry shows that for $1<\varepsilon<2/\sqrt{3}$ the set $\sigma_{\varepsilon}(A)$ is $N-1$ connected.
In section~\ref{sec: local minima} we give an example with a non-normal matrix, see Figure~\ref{fig:pseudo} in subsection~\ref{sec: block diagonals}.
\end{remark}

\section{Examples}%
\label{sec: local minima}
In this section we give a number of examples. 
In subsection~\ref{ex-assum} we give an example of an operator with an eigenvalue $\lambda_{\rm max}(z)$ that has infinite multiplicity. In this example Assumption~\ref{assum} holds for some but not all $z\in\rho(A)$. 
In subsections~\ref{sec: block diagonals}  and~\ref{sec: general example} we give examples of non-normal matrices such that the resolvent norm has a local minimum at the origin. In subsection~\ref{infinite} we give an example of a non-normal operator in an infinite dimensional Hilbert space which satisfies Assumption~\ref{assum} at the origin and such that the resolvent norm has a local minimum at the origin. 

\subsection{On Assumption \ref{assum}}
\label{ex-assum}
We give examples showing that Assumption~\ref{assum} may only be satisfied in a proper subset of the resolvent set.

For $N\in\IN$ let $\cH_N=L^2([0,1])\oplus \IC^N$ and for $N=\infty$ let
$\cH_{\infty}=L^2([0,1])\oplus\ell^2(\IN)$. Let $A_1$ be multiplication by $x$ on $L^2([0,1])$, such that $\sigma(A_1)=\sigma_{\rm ac}(A_1)=[0,1]$, and let $A_2=2I$ on $\IC^N$ or $\ell^2(\IN)$. Let $A=A_1\oplus A_2$ on $\cH_N$, $N\in\IN\cup\{\infty\}$. Then $A$ is a bounded self-adjoint operator with $\sigma(A)=[0,1]\cup\{2\}$. Using $\norm{S(z)}=\norm{R_A(z)}^2=1/\dist(z,\sigma(A))^2$ we see that Assumption~\ref{assum} is satisfied in
$\cG=\{z\in\IC\,|\,\re z>\frac32, z\neq2\}$. It is not satisfied in $\rho(A)\setminus\cG$. For $z\in\cG$ we have $\lambda_{\rm max}(z)=\abs{2-z}^{-2}$ with multiplicity $N$.

\subsection{Local minimum for block diagonals with $\mathbf{2\times 2}$-blocks}%
\label{sec: block diagonals}

Let $A \in \IC^{2\times 2}$ have distinct eigenvalues $\frac{\tr(A)}{2}\pm r_A \e^{\I \phi}$ for $r_A>0$. 
By \eqref{eq:resnormA2} and \eqref{eq:resnormtype2} we have
\begin{equation*}
\norm{R_{A}(\tr(A)/2)}^2 = \frac{2{r_A}^{-2}}{k_A-\sqrt{{k_A}^2-4}} = \frac{1}{r_A^2\gamma_A}
\end{equation*}
for $\gamma_A = (k_A-\sqrt{{k_A}^2-4})/2 \in~]0,1]$ and some $k_A\geq 2$ as given in the proof of Theorem~\ref{thm-henrik}(ii) (any real value $\geq 2$ may be obtained, depending on the structure of the matrix). 
Note that $k_A$ is independent of the eigenvalues of $A$.

\begin{lemma} \label{lemma3}
Let $A\in \IC^{2\times 2}$ with distinct eigenvalues $\frac{\tr(A)}{2}\pm r_A \e^{\I\phi}$ for $r_A>0$ and define
\begin{equation*}
\theta_A = \frac{\pi}{2} - \frac{1}{2}\arccos(\gamma_A) \in~]\tfrac{\pi}{4},\tfrac{\pi}{2}].
\end{equation*}
The angles for which $\norm{R_A(\cdot)}$ is increasing away from $\tr(A)/2$ {\rm(}in a small neighborhood{\rm)} are precisely the two arcs $]\phi-\theta_A,\phi+\theta_A[\;\cup\;]\phi+\pi-\theta_A,\phi+\pi+\theta_A[$.
\end{lemma}

\begin{proof}
For simplicity consider $A_2 = {r_A}^{-1} \e^{-\I\phi}(A-\frac{\tr(A)}{2}I)$ such that $A_2$ is as in the proof of Theorem~\ref{thm-henrik}(ii). 
	
For $k_A=2$ let $z=\sqrt{t} \e^{\I\theta}$ as in the proof of Theorem~\ref{thm-henrik}(ii).
We arrive again at \eqref{eq:deriv} under the condition that $t>0$ and $\cos(2\theta)\neq -1$ (avoiding the imaginary axis where we know $\norm{R_{A_2}}$ is decreasing), i.e.
\begin{equation*}
g(2,\theta) = \frac{1+\cos(2\theta)}{2} > 0,\quad \theta\in~]-\pi,\pi]\setminus \{\pm\tfrac{\pi}{2}\}.
\end{equation*}
So for $k_A=2$ the resolvent norm $\norm{R_{A_2}(\cdot)}$ is increasing away from $z=0$ (in a neighborhood) in all directions except along the imaginary axis where it decreases, corresponding to $\theta_A = \frac{\pi}{2}$. 
	
Now assume $k_A>2$. Thus $\gamma_A\in~]0,1[$ and therefore 
$\theta_A\in~]\tfrac{\pi}{4},\frac{\pi}{2}[$. 
By \eqref{eq:deriv} it holds for $\theta_0\in~]-\pi,\pi]$ that $g(k_A,\theta_0)=0$, or equivalently $\cos(2\theta_0)=-\gamma_A$, if and only if
\begin{equation*}
\theta_0 = \begin{cases}
		\phantom{-}\frac{\pi}{2}\pm \frac{1}{2}\arccos(\gamma_A) \\
		-\frac{\pi}{2}\pm \frac{1}{2}\arccos(\gamma_A)
		\end{cases} = \begin{cases}
			\pm \theta_A, \\
			\pm \pi \mp \theta_A.
		\end{cases}
\end{equation*}
From the proof of Theorem~\ref{thm-henrik} we know that $\norm{R_{A_2}(\cdot)}$ is increasing away from $z=0$ on the real line (in a neighborhood) and decreasing away from $z=0$ on the imaginary line. 
Hence, \eqref{eq:deriv} implies that $g(k_A,\theta)>0$ for $\theta\in~]-\theta_A,\theta_A[\;\cup\; ]\pi-\theta_A,\pi+\theta_A[$ and $g(k_A,\theta) \leq 0$ for $\theta\in [\theta_A,\pi-\theta_A]\cup[\pi+\theta_A,2\pi -\theta_A]$. 
	
To recapitulate, for a small enough neighborhood of $0$, the directions for which $\norm{R_{A_2}(\cdot)}$ is increasing away from the origin are precisely the angles $]-\theta_A,\theta_A[\;\cup\; ]\pi-\theta_A,\pi+\theta_A[$. 
Now the rotation by $\e^{\I\phi}$ occurring in \eqref{eq:resnormtype2} concludes the proof.
\end{proof}

From Lemma~\ref{lemma3} we can conclude that if $A\in \IC^{2\times 2}$ has eigenvalues $z_0 \pm r_A \e^{\I\phi}$ with $r_A>0$, then $\norm{R_A(\cdot)}$ is increasing near $z_0$ at least in a $\frac{\pi}{2}$-arc centered at $\phi$ and in a $\frac{\pi}{2}$-arc centered at $\phi+\pi$. 
Thus, we directly get the following result by constructing a block diagonal matrix, where for each direction from a point $z_0$ there is at least one block for which the resolvent norm is increasing. We denote the torus by $\IT = \{z\in \IC\,|\, \abs{z}=1 \}$.
\begin{theorem}\label{thm: henrik 2}
Let $\{A_j\}_{j=1}^M\subset \IC^{2\times 2}$ such that $A_j$ has eigenvalues $z_0\pm r_j \e^{\I\phi_j}$ for $r_j>0$. 
Let $B = \oplus_{j= 1}^M A_j$ such that $B$ has the eigenvalues of each $A_j$ and
\begin{equation*}
\norm{R_{B}(\cdot)} = \max_{j=1,\dots,M}\norm{R_{A_j}(\cdot)}.
\end{equation*}
Then $\norm{R_{B}(\cdot)}$ has a local minimum at $z_0$ if and only if there is a subset of indices $J\subseteq \{1,\dots,M\}$ such that
\begin{enumerate}[\rm(i)]
\item $\norm{R_{A_{j'}}(z_0)} < K = \norm{R_{A_{j}}(z_0)},
\quad j\in J,\quad  j'\not\in J$,
\item $\IT = \{\e^{\I\phi} \,|\, \phi\in \Phi\cup(\pi+\Phi)\}$, where $\Phi = \bigcup_{j\in J} ]\phi_j-\theta_{A_j},\phi_j +\theta_{A_j}[$.
\end{enumerate}
\end{theorem}

\begin{proof}
We note that (i) implies, by continuity, that the value of $\norm{R_{B}(\cdot)}$ near $z_0$ is determined by $\{A_j\}_{j\in J}$. 
(ii) is a necessary condition since otherwise there is a direction for which all $\norm{R_{A_j}(\cdot)}$ are decreasing near $z_0$ for $j\in J$.  
	
On the other hand, if (i) and (ii) hold then for each direction near $z_0$ we have that $\norm{R_{B}(\cdot)}$ equals $\max_{j\in J}\norm{R_{A_j}(\cdot)}$ for which at least one of the resolvent norms $\norm{R_{A_j}(\cdot)}$ is increasing. Since $\{\norm{R_{A_j}(\cdot)}\}_{j\in J}$ coincide at $z_0$ then $\max_{j\in J} \norm{R_{A_j}(\cdot)}$ is increasing near $z_0$.
\end{proof}
As a sufficient condition in Theorem~\ref{thm: henrik 2} we may simply have $\norm{R_{A_j}(z_0)} = K$,  $j=1,\dots,M$, for a constant $K>0$ and
\begin{equation*}
\IT = \left\{\e^{\I\phi} \,|\, \phi\in \bigcup{}_{j=1}^M \left( [\phi_j-\tfrac{\pi}{4},\phi_j +\tfrac{\pi}{4}]\cup[\phi_j+\tfrac{3\pi}{4},\phi_j+\tfrac{5\pi}{4}] \right)\right\}.
\end{equation*}

\begin{example}
Let $A_1,A_2\in\IC^{2\times 2}$ have eigenvalues, respectively $z_0\pm r_{A_1} \e^{\I\phi}$ and $z_0\pm r_{A_2} \I\e^{\I\phi}$ for $r_{A_1},r_{A_2}>0$. 
I.e. the eigenvalues are placed on orthogonal lines intersecting at $z_0$. 
For $B=A_1\oplus A_2$ then $\norm{R_{B}(\cdot)}$ has a local minimum at $z_0$ if and only if $\norm{R_{A_1}(z_0)} = \norm{R_{A_2}(z_0)}$.
This is the case if $\gamma_{A_1}/\gamma_{A_2} = \left({r_{A_2}}/{r_{A_1}}\right)^2$.
\end{example}

\begin{example} \label{ex:last}
Consider the following matrices
\begin{equation*}
A_1 = \begin{bmatrix}
		\I & 0 \\ 0 & -\I
      \end{bmatrix}, \quad
A_2 = \sqrt{\tfrac{2}{6-\sqrt{32}}}\,\e^{-\frac{\pi}{6}\I}\begin{bmatrix}
			1 & 2 \\ 0 & -1
		\end{bmatrix}, \quad A_3 = \sqrt{\tfrac{2}{30-\sqrt{896}}}\,\e^{\frac{\pi}{6}\I} \begin{bmatrix}
			2\I & -1 \\ -5 & -2\I
		\end{bmatrix}.
\end{equation*}
These are all scaled and rotated versions of matrices of the type \eqref{eq:type2} such that they satisfy the conditions of Theorem~\ref{thm: henrik 2} with $z_0 = 0$, $\phi_1 = \frac{\pi}{2}$, $\phi_2= -\frac{\pi}{6}$ and $\phi_3 = \frac{\pi}{6}$. 
Moreover, we have $\norm{R_{A_1}(0)} = \norm{R_{A_2}(0)} = \norm{R_{A_3}(0)} = 1$. 
Thus with $B = A_1\oplus A_2\oplus A_3$ we get that $\norm{R_{B}(\cdot)}$  has a local minimum at the origin. Note that $B$ is not normal.

\begin{figure}[htb]
		\centering
		\includegraphics[width=.6\textwidth]{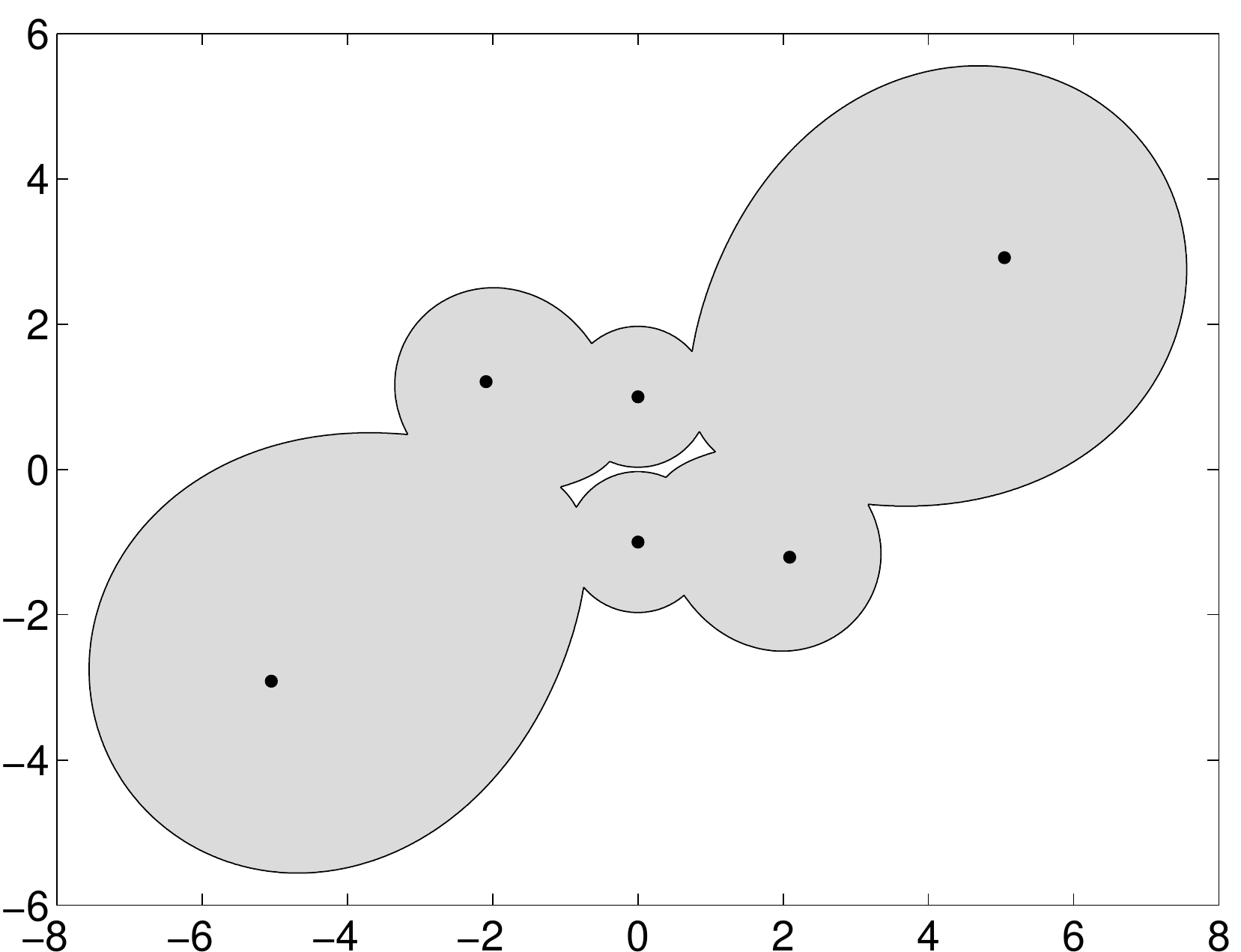}
		\caption{Pseudospectrum $\sigma_\varepsilon(B)$ with $\varepsilon = 0.97$ and eigenvalues of $B$ (black dots) from Example~\ref{ex:last}.} 
                \label{fig:pseudo}
\end{figure}

This is furthermore observed in Figure~\ref{fig:pseudo} where for appropriate values of $\varepsilon$, $\sigma_\varepsilon(B)$ is a connected set, but not simply connected.
In particular it excludes a region near the origin.
\end{example}

\subsection{Examples for any finite $\mathbf{\textit N > 2}$}%
\label{sec: general example}

Let $N>2$ and $\{a_j\}_{j=1}^N \subset \IC\setminus\{0\}$ such that $\abs{a_1} > \abs{a_j}$ for $j=2,\dots,N$. Define
\begin{equation}
	A = \begin{bmatrix}
		0 & \ldots & \ldots & 0 & \frac{1}{a_1} \\
		\frac{1}{a_2} & \ddots &  & & 0 \\
		 & \frac{1}{a_3} & \ddots & & \vdots \\
		 & & \ddots & \ddots & \vdots\\
		 & & & \frac{1}{a_N} & 0
	\end{bmatrix}. \label{eq:genexample}
\end{equation}
Then we have:
\begin{equation*}
	A^{-1} = \begin{bmatrix}
			0 & a_2 &  &  & \\
			\vdots & \ddots & a_3 & &  \\
			\vdots &  & \ddots & \ddots &  \\
			0 & &  & \ddots & a_N\\
			a_1 & 0 & \ldots & \ldots & 0
		\end{bmatrix}, \quad S(0) = (A^{-1})^{\ast}A^{-1} = \begin{bmatrix}
					 \abs{a_1}^2 & & & & \\
					 & \abs{a_2}^2 & & &  \\
					 & & \abs{a_3}^2 & & \\
					 & & & \ddots      &\\
					 & & & & \abs{a_N}^2
		\end{bmatrix}.
\end{equation*}
Thus the normalized eigenvectors for $S(0)$ corresponding to the largest eigenvalue $\abs{a_1}^2$ are 
\begin{equation*}
	\psi = (\e^{\I\theta},0,\dots,0),\quad \theta\in\IR.
\end{equation*}
Furthermore, as $N>2$ we have
\begin{equation*}
	A^{-2} = \begin{bmatrix}
			0 & 0 & a_2a_3 &  &  & \\
			\vdots & \ddots & \ddots & a_3a_4 &  \\
			\vdots &  & \ddots & \ddots & \ddots&  \\
			0 &  &  & \ddots & \ddots & a_{N-1}a_N\\
			a_1a_N & 0 &  &  & \ddots& 0 \\
			0 & a_1a_2 & 0 & \ldots & \ldots & 0 
		\end{bmatrix}.
\end{equation*}
Thus we have
\begin{equation*}
	\ip{\psi}{A^{-1}\psi} = \ip{\psi}{A^{-2}\psi} = 0,
\end{equation*}
which by Corollary~\ref{cor12} implies that $\norm{R_A(\cdot)}$ has a local minimum at the origin.

Note that this example fails for $N=2$, since in that case $A^{-2} = \diag(a_1a_2,a_1a_2)$, so $\ip{\psi}{A^{-2}\psi} \neq 0$.

See Figure~\ref{fig:pseudo2} for a specific choice of $\{a_j\}_{j=1}^6$ for $N=6$.

\begin{figure}[htb]
	\centering
	\includegraphics[width=.5\textwidth]{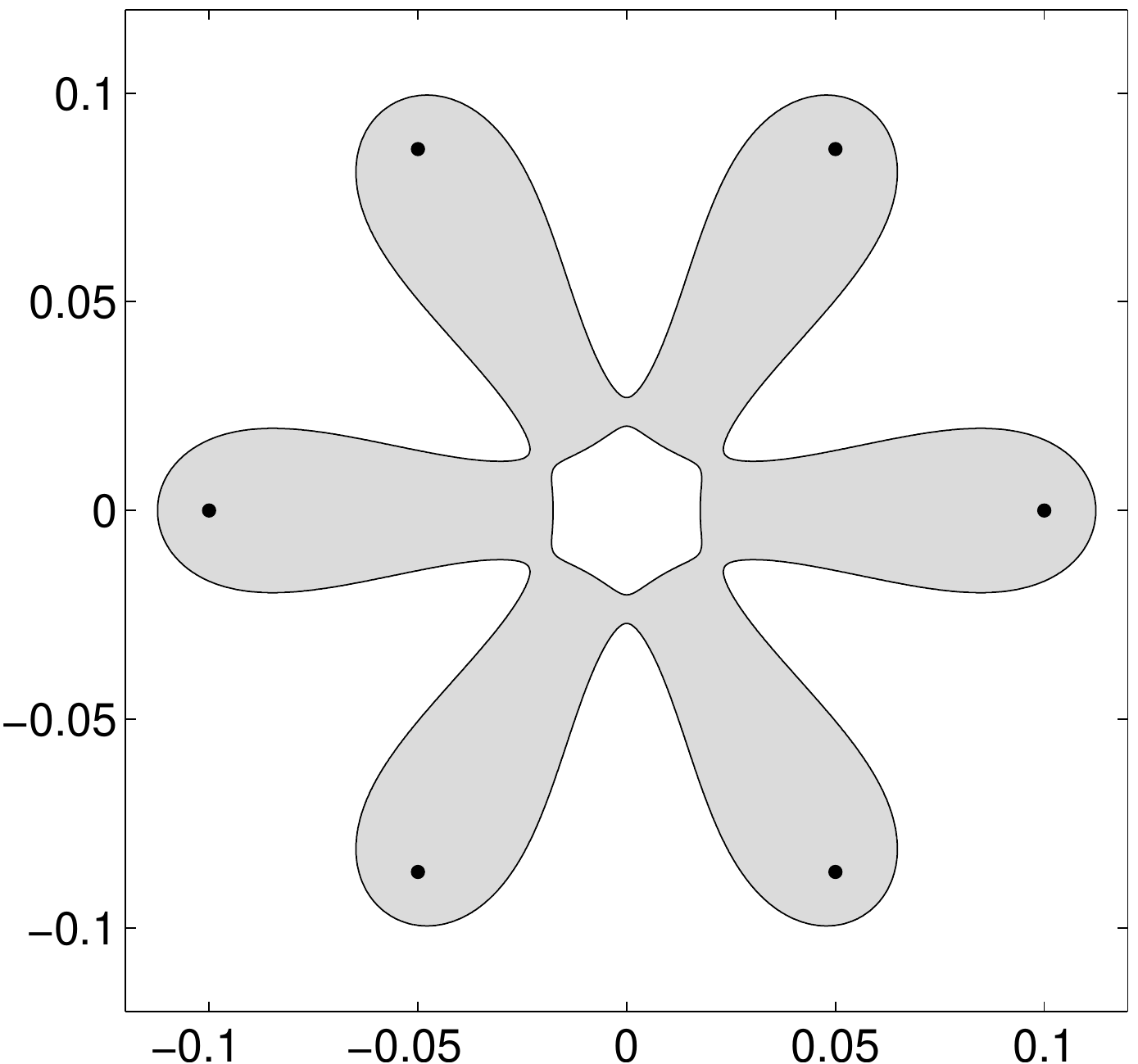}
	\caption{Pseudospectrum $\sigma_\varepsilon(A)$ with $\varepsilon = 9.9966\cdot 10^{-7}$ and eigenvalues of $A$ (black dots) of the matrix in \eqref{eq:genexample} where $N=6$, $a_1 = 10^6$ and $a_j = 1$ for $j = 2,\dots,6$.} 
	\label{fig:pseudo2}
\end{figure}

\subsection{An infinite dimensional example}
\label{infinite}
We give an example of a non-normal operator on an infinite dimensional Hilbert space sa\-tis\-fy\-ing Assumption~\ref{assum} such that its resolvent norm has a local minimum at the origin.

Let $\cH=\ell^2(\IZ)$. Let $a_j\in\IC\setminus\{0\}$,  $j\in\IZ$, be a sequence which satisfies
\begin{equation}\label{a-assum}
\abs{a_0} > \sup_{j \neq 0} \abs{a_j}. 
\end{equation}
Define an operator $A$ by
\begin{equation*}
(Ax)_j=a_{j+1}^{-1}x_{j+1},\quad x\in D(A)
\end{equation*}
where $D(A)=\{x\in\cH\,|\, Ax\in \cH\}$ is the maximal domain. It is easy to verify that $A$ is densely defined and closed. We have that $A$ is invertible with a bounded inverse
\begin{equation*}
(A^{-1}x)_j=a_jx_{j-1},\quad x\in\cH.
\end{equation*}

A computation shows that $(A^{-1})^{\ast}A^{-1}$ is given by $((A^{-1})^{\ast}A^{-1}x)_j=\abs{a_{j+1}}^2x_j$, $j\in\IZ$. Thus a basis of eigenvectors is given by the canonical basis $\{e_j\}_{j\in\IZ}$. By construction the largest eigenvalue is simple and equals $\abs{a_0}^2$, with $\psi=e_{-1}$ as a normalized eigenvector. We have
\begin{align*}
\ip{\psi}{A^{-1}\psi}&=a_0\ip{e_{-1}}{e_0}=0,\\
\ip{\psi}{A^{-2}\psi}&=a_0a_1\ip{e_{-1}}{e_{1}}=0.
\end{align*}
Thus by Corollary~\ref{cor12} $\norm{R_A(\cdot)}$ has a local minimum at the origin.

We note that if $\lim_{\abs{j}\to\infty}a_j=0$, then $A$ has compact resolvent. If $\inf_{j\in\IZ}\abs{a_j}>0$ then $A$ is bounded.
 
\subsection*{Acknowledgements}
HC, AJ, and HKK were partially  supported by the Danish Council for Independent Research $|$ Natural Sciences, Grant DFF--4181-00042. AJ thanks Kenji Yajima, Gakushuin University, Tokyo, Japan, for useful comments.

\end{document}